\documentclass[10pt,a4paper,reqno]{amsart}
\usepackage{amsmath,amsthm}
\usepackage{color,enumerate}
\usepackage{amssymb}
\usepackage{upgreek}
\usepackage{mathrsfs}
\usepackage{mathabx}
\usepackage{txfonts}
\usepackage{enumitem}		
\usepackage{soul}
\usepackage[colorlinks=true]{hyperref}	

\makeatletter \@addtoreset{equation}{section} \makeatother
\renewcommand\thetable{\thesection.\@arabic\c@table}

\theoremstyle{plain}
\newtheorem{maintheorem}{Theorem}

\newtheorem{theorem}{Theorem }[section]
\newtheorem{proposition}[theorem]{Proposition}
\newtheorem{lemma}[theorem]{Lemma}
\newtheorem{corollary}[theorem]{Corollary}

\theoremstyle{definition} \theoremstyle{remark}
\newtheorem{remark}[theorem]{Remark}

\newtheorem{definition}[theorem]{Definition}

\newcommand{\supp}{\operatorname{supp}}
\newcommand{\cH}{\mathcal{H}}

\newcommand{\cR}{\mathcal{R}}

\newcommand{\cO}{\mathcal{O}}

\newcommand{\cN}{\mathcal{N}}

\newcommand{\arxiv}[1]{Preprint \href{http://arxiv.org/abs/#1}{arXiv:{#1}}}

\begin{document}

\title{Positivity of the top Lyapunov exponent for cocycles on semisimple Lie groups over hyperbolic bases}

\author{M.~Bessa}
\address{M\'ario Bessa: Departamento de Matem\'atica, Universidade da Beira Interior, Rua Marqu\^es d'\'Avila e Bolama,
  6201-001 Covilh\~a,
Portugal.}
\email{bessa@ubi.pt}
\urladdr{http://cmup.fc.up.pt/cmup/bessa/}

\author{J.~Bochi}
\address{Jairo Bochi: Facultad de Matem\'aticas, Pontificia Universidad Cat\'olica de Chile, Av.\ Vicu\~na Mackenna 4860 Santiago, Chile.}
\email{jairo.bochi@mat.uc.cl}
\urladdr{http://www.mat.uc.cl/$\sim$jairo.bochi/}

\author{M.~Cambrainha}
\address{Michel Cambrainha: Departamento de Matem\'atica e Estat\'istica, Universidade Federal do Estado do Rio de Janeiro,
Av.\ Pasteur 458, 22290-240 Rio de Janeiro, Brazil.}
\email{michel.cambrainha@uniriotec.br}

\author{C.~Matheus}
\address{Carlos Matheus: Universit\'e Paris 13, Sorbonne Paris Cit\'e, LAGA, CNRS (UMR 7539), F-93439, Villetaneuse, France}
\email{matheus.cmss@gmail.com}

\author{P.~Varandas}
\address{Paulo Varandas: Departamento de Matem\'atica, Universidade Federal da Bahia\\
Av.\ Ademar de Barros s/n, 40170-110 Salvador, Brazil.}
\email{paulo.varandas@ufba.br}
\urladdr{www.pgmat.ufba.br/varandas}

\author{D.~Xu}
\address{Disheng Xu: Universit\'e Paris Diderot, Sorbonne Paris Cit\'e, Institut de Math\'ematiques de Jussieu - Paris Rive gauche, UMR 7586, CNRS, Sorbonne Universit\'es, UPMC Universit\'e Paris 06, F-75013, Paris, France.}
\email{disheng.xu@imj-prg.fr}
\urladdr{https://webusers.imj-prg.fr/disheng.xu}

\thanks{Bessa is partially supported by FCT - `Funda\c c\~ao para a Ci\^encia e a Tecnologia', through Centro de Matem\'atica e Aplica\c c\~oes (CMA-UBI), Universidade da Beira Interior, project UID/MAT/00212/2013.
Bochi is partially supported by project Fondecyt 1140202 (Chile). 
Varandas is partially supported by CNPq-Brazil.}

\date{\today}

\begin{abstract}
A theorem of Viana says that almost all cocycles over any hyperbolic system have nonvanishing Lyapunov exponents. 
In this note we extend this result to cocycles on any noncompact classical semisimple Lie group.
\end{abstract}

\maketitle

\section{Introduction}

Lyapunov exponents are ubiquitous in differentiable dynamics \cite{BP}, control theory \cite{CK}, one-dimensional Schr\"odinger operators \cite{Bourgain}, random walks on Lie groups \cite{Furman}, among other fields.
Let us recall the basic definition.
Let $(M,\mu)$ be a probability space, and let $f \colon M \to M$ be a measure-preserving discrete-time dynamical system.
Let $A \colon M \to \mathbb{R}^{d \times d}$ be a at least measurable matrix-valued map.
The pair $(A,f)$ is called a \emph{linear cocycle}.
We form the products:
\begin{equation}\label{e.product}
A^{(n)}(x) \coloneqq A(f^{n-1}(x)) \cdots A(f(x)) A(x) \, .
\end{equation}
Let $\| \mathord{\cdot} \|$ be any matrix norm, and assume that $\log^+\|A\|$ is $\mu$-integrable.
The \emph{(top) Lyapunov exponent} of the cocycle is 
\begin{equation}\label{e.top_LE}
\lambda_1(A,f,x) \coloneqq \lim_{n \to \infty} \frac{1}{n} \log \big\| A^{(n)}(x) \big\|,
\end{equation}
which by the subadditive ergodic theorem 
is well-defined (possibly $-\infty$) for $\mu$-almost every $x$, and is independent of the choice of norm.
If $\mu$ is ergodic, then the Lyapunov exponent is almost everywhere equal to a constant, which we denote by $\lambda_1(A,f,\mu)$.

The Lyapunov exponent is a very subtle object of study. 
Let us explain the type of question we are interested in. 
Consider maps $A$ taking values in the group $\mathrm{SL}(d,\mathbb{R})$.
In that case, the Lyapunov exponent is nonnegative, and it is reasonable to expect that it should be  positive except in some degenerate or fragile situations.  
As a result in this direction, Knill \cite{Knill} proved that for any base dynamics $(f,\mu)$ where the measure $\mu$ is ergodic and non-atomic, $\lambda_1(A,f,\mu) > 0$ for all maps $A$ in a dense subset of the space $L^\infty(M, \mathrm{SL}(2,\mathbb{R}))$.
Still in $d=2$, this result was extended to virtually any regularity class (continuous, H\"older, smooth, analytic) by Avila \cite{Avila}.
The case $d>2$ remains unsolved, though similar results have been obtained by Xu \cite{Xu} for some other matrix groups as the symplectic groups. 
In general, the sets of maps where the Lyapunov exponents are positive are believed to be not only dense, but also ``large'' in a probabilistic sense (see \cite{Avila}).
However, in low regularity as $L^\infty$ or $C^0$, these sets can be ``small'' in a topological sense; indeed they can be locally meager \cite{B,BV}. 

Historically, the first case to be studied was random products of i.i.d.\ matrices, which fits in the general setting of linear cocycles by taking $(f,\mu)$ as the appropriate Bernoulli shift on $\mathrm{SL}(d,\mathbb{R})^\mathbb{Z}$, and the matrix map $A$ depending only on the zeroth coordinate.
Furstenberg showed that that the Lyapunov exponent is positive under explicit mild conditions (Theorem 8.6 in \cite{Furstenberg}). Finer results were later obtained (still in the i.i.d.\ case) by Guivarc'h and Raugi \cite{GR}, Gol'dsheid and Margulis \cite{GM}, among others.

As first shown in Ledrappier's seminal paper \cite{L}, and later vigorously expressed in the work of Viana and collaborators \cite{BoGV03,BoV04,AV07,Viana,AV10,ASV}, the philosophy of random i.i.d. products of matrices can be adapted to other contexts, where Bernoulli shifts are replaced by more general classes of dynamical systems with hyperbolic behavior, at least under certain conditions on the maps $A$. A landmark result, proved by Viana in \cite{Viana}, can be stated informally as follows: If the dynamics $(f,\mu)$ is nonuniformly hyperbolic (and satisfies an additional technical but natural hypothesis) then, in spaces of sufficiently regular (at least H\"older) maps $A \colon M \to \mathrm{SL}(d, \mathbb{K})$ for $\mathbb{K}=\mathbb{R}$ or $\mathbb{C}$, positivity of the Lyapunov exponent occurs on a set which is large in both a topological and in a probabilistic sense. 

In this note we show that the groups $\mathrm{SL}(d, \mathbb{R})$ and $\mathrm{SL}(d, \mathbb{C})$ in Viana's theorem can be replaced by any noncompact classical semisimple group of matrices, as for example the symplectic groups, pseudo-unitary groups, etc. This provides an answer to a question of Viana \cite[Problem~4]{Viana}.

Let us mention that when $f$ is quasiperiodic (and so lies in a region antipodal to hyperbolicity in the dynamical universe), the study of Lyapunov exponents forms another huge area of research: see for instance \cite{AK,DK} and references therein.

Finally, we note that for \emph{derivative cocycles} (i.e., where $A = Df$) very few general results are known, except in low topologies  \cite{B,BV,ACW}.

\section{Precise setting}\label{s.statements}

In this section we recall some basic notions about multiplicative ergodic theory, and then state our results.
The reader is referred to~\cite{BP,Viana} for more details and references.

\subsection{Lyapunov exponents}\label{ss.lyapunov}

The top Lyapunov exponent of a cocycle $(A,f)$ was defined in \eqref{e.top_LE}.
In general, we define Lyapunov exponents $\lambda_1(A,f,x) \ge \lambda_2(A,f,x) \ge \cdots \ge \lambda_d(A,f,x)$
by
$$
\lambda_i(A,f,x) \coloneqq \lim_{n \to \infty} \frac{1}{n} \log \sigma_i(A^{(n)}(x)) \, ,
$$
where $\sigma_i(\mathord{\cdot})$ denotes the $i$-th singular value.

\subsection{Hyperbolic measures and local product structure}\label{ss.hyperbolic-systems}

Let $f:M\to M$ be a $C^{1+\alpha}$ diffeomorphism of a compact manifold $M$,
and let $\mu$ be a invariant Borel probability measure.
Suppose that $\mu$ is \emph{hyperbolic}, that is, the Lyapunov exponents of the derivative cocycle $Df$ are all different from zero at $\mu$-almost every point $x$.
So, by Oseledets theorem, we can split the tangent bundle $T_x M$ as the sum of the subspaces  $E^\mathrm{u}_x$ and $E^\mathrm{s}_x$ corresponding to positive and negative exponents, respectively.

Given a hyperbolic probability measure $\mu$, Pesin's stable manifold theorem (see e.g.~\cite{BP}) says that, for $\mu$-almost every $x$, there exists a
$C^1$-embedded disk $W_{\mathrm{loc}}^\mathrm{s}(x)$ (local stable manifold at $x$) such that
$T_x  W_{\mathrm{loc}}^\mathrm{s}(x)=E^\mathrm{s}_x$, it is forward invariant $f(W_{\mathrm{loc}}^\mathrm{s}(x))\subset W_{\mathrm{loc}}^\mathrm{s}(f(x))$
and the following holds: given 
$0 < \tau_x < |\lambda_{1 + \dim E^\mathrm{u}_x} (Df,f,x)|$,
there exists $K_x>0$ such that $d(f^n(y),f^n(z))\le K_x \,e^{-n\tau_x} d(y,z)$ for every $y,z\in W_{\mathrm{loc}}^\mathrm{s}(x)$.
Local unstable manifolds $W_{\mathrm{loc}}^\mathrm{u}(x)$ are defined analogously using $E^\mathrm{u}_x$
and $f^{-1}$.

Moreover, since local invariant manifolds and the constants above vary measurably with the point $x$ one can select \emph{hyperbolic blocks} $\cH(K,\tau)$ 
in such a way that $K_x\le K$ and $\tau_x\ge \tau$ for all $x\in \cH(K,\tau)$, 
the local manifolds $W_{\mathrm{loc}}^\mathrm{s}(x)$ and $W_{\mathrm{loc}}^\mathrm{u}(x)$ vary continuously with $x\in \cH(K,\tau)$; moreover, $\mu(\cH(K,\tau)) \to 1$ as $K \to \infty$ and $\tau \to 0$.
In particular, if $x\in \cH(K,\tau)$ and $\delta>0$ is small enough, then for every $y$, $z\in B(x,\delta) \cap \cH(K,\tau)$,
the intersection $W^\mathrm{u}_{\mathrm{loc}}(y)\cap W^\mathrm{s}_{\mathrm{loc}}(z)$ is transverse and  consists of a unique point, denoted $[y,z]$.

For each $x \in \cH(K,\tau)$, define sets:
\begin{align*}
\cN_x^\mathrm{u}(\delta) \coloneqq \{ [x,y] \in  W^\mathrm{u}_{\mathrm{loc}}(x) : y \in \cH(K,\tau) \cap B(x,\delta) \} \, , \\
\cN_x^\mathrm{s}(\delta)\coloneqq \{ [y,x] \in  W^\mathrm{s}_{\mathrm{loc}}(x) : y \in \cH(K,\tau) \cap B(x,\delta) \} \, . 
\end{align*}
Let $\cN_x(\delta)$ be the image of $\cN_x^\mathrm{u}(\delta) \times \cN_x^\mathrm{s}(\delta)$ under  the map $[ \mathord{\cdot} , \mathord{\cdot}]$. This is a small ``box'' neighborhood of $x$ in the block $\cH(K,\tau)$, and (reducing $\delta$ if necessary) the following map is a homeomorphism:
$$
\begin{array}{rcl}
\Upsilon_x : \cN_x(\delta)& \to & \cN_x^\mathrm{u}(\delta) \times \cN_x^\mathrm{s}(\delta) \\
		 y & \mapsto & ( [x,y] , [y,x]  )
\end{array}
$$

\begin{definition}[{\cite[p.~646]{Viana}}]\label{def:lps}
The hyperbolic measure $\mu$ has \emph{local product structure} if for every $(K,\tau)$, every small $\delta>0$ as before, and every $x\in \cH(K,\tau)$, the measure $\mu\!\mid_{\cN_x(\delta)}$ is equivalent to the product measure $\mu^\mathrm{u}_x \times \mu^\mathrm{s}_x$, where
$\mu_x^{i}$ denotes the conditional measure of $(\Upsilon_x)_*(\mu\!\mid_{\cN_x(\delta)})$ on  $\cN^{i}_x(\delta)$, for $i\in\{\mathrm{u},\mathrm{s}\}$.
\end{definition}

\subsection{Space of cocycles}

The relevant functional spaces of linear cocycles for our subsequent discussion are defined as follows. 
Let $G$ be a Lie subgroup of $\mathrm{GL}(d,\mathbb{C})$, 
let $M$ be a Riemannian compact manifold $M$,
and let $(r,\nu) \in \mathbb{N} \times [0,1] - \{(0,0)\}$.
Let $C^{r,\nu}(M, G)$ denote the set of maps $A:M \to G$ of class $C^r$ such that $D^r A$ is $\nu$-H\"older continuous if $\nu>0$. 
We equip this set with the topology induced by the distance:
$$
d_{r,\nu} (A,B) \coloneqq \underset{0\le j \le r} {\sup}\|D^j(A-B)(x)\| +\underset{x\not=y}{\sup}\frac{\|D^r(A-B)(x)-D^r(A-B)(y)\|}{d(x,y)^\nu} \, ,
$$
where the last term is omitted if $\nu = 0$.
Then $C^{r,\nu}(M, G)$ is a Banach manifold.

\subsection{Statement of the results}\label{ss:statement}

Let $\mathbb{K}$ be either $\mathbb{R}$ or $\mathbb{C}$, and let $d \ge 2$.
Let $\mathbf{G}$ be a $\mathbb{K}$-algebraic subgroup of $\mathrm{SL}(d,\mathbb{C})$, i.e., a group of complex $d\times d$ matrices of determinant $1$, defined by polynomial equations with coefficients in $\mathbb{K}$.
Denote $G \coloneqq \mathbf{G} \cap \mathrm{SL}(d,\mathbb{K})$.
Henceforth we will assume the following properties:
\begin{enumerate}
\item\label{i.connected} 
$\mathbf{G}$ is connected (or equivalently, $\mathbf{G}$ is irreducible as an algebraic set);
\item\label{i.semisimple}  
$G$ is semisimple (or equivalently, $\mathbf{G}$ is semisimple);
\item\label{i.noncompact}  
$G$ is noncompact;
\item\label{i.irreducible}
$G$ acts irreducibly on $\mathbb{K}^d$, that is, the sole subspaces $V\subset \mathbb{K}^d$ invariant under the whole action of $G$ are the trivial subspaces $V=\{0\}$ and $V=\mathbb{K}^d$.
\end{enumerate}
Our assumptions are satisfied by all noncompact classical groups $G$, that is, $\mathrm{SL}(d, \mathbb{K})$ for $d\geq 2$, $\mathrm{SL}(n,\mathbb{H}) \simeq \mathrm{SU}^*(2n)$, $\mathrm{Sp}(n, \mathbb{K})$, $\mathrm{Sp}(n,m)$ for $n, m\geq 1$, $\mathrm{SO}(m,n)$ for $m, n\geq 1$, $m+n\geq 3$, $\mathrm{SU}(m,n)$ for $m, n\geq 1$, and $\mathrm{SO}^*(2n)$ for $n\geq 2$.

The following result is exactly Theorem~A in \cite{Viana} when $G = \mathrm{SL}(d,\mathbb{K})$:

\begin{maintheorem}\label{Viana1}  
Let $G$ be a group of matrices satisfying the hypotheses above. 
Let $f$ be a $C^{1+\alpha}$-diffeomorphism of a compact manifold $M$.
Let $\mu$ be a $f$-invariant ergodic hyperbolic non-atomic probability measure with local product structure.
Let $(r,\nu) \in \mathbb{N} \times [0,1] - \{(0,0)\}$.
Then there exists an open and dense subset $\mathscr{G}$ of $C^{r,\nu}(M, G)$ such that for any $A\in\mathscr{G}$, the cocycle $(A,f)$ has at least one positive Lyapunov exponent
at $\mu$-a.e.\ point.
Moreover, the complement of $\mathscr{G}$ in $C^{r,\nu}(M, G)$ has infinite codimension.
\end{maintheorem}

The last statement means that the complement of $\mathscr{G}$ is locally contained in Whitney stratified sets (see \cite{strat-2}) of arbitrarily large codimension. In particular, $\mathscr{G}$ is large in a very strong probabilistic sense.

Arguing exactly as in \cite[p.~676]{Viana}, we obtain the following consequence in the non-ergodic case:

\begin{corollary}\label{c.Viana1} 
Let $G$ be a group of matrices satisfying the hypotheses above. 
Let $f$ be a $C^{1+\alpha}$-diffeomorphism of a compact manifold $M$.
Let $\mu$ be a $f$-invariant ergodic hyperbolic non-atomic probability measure with local product structure.
Let $(r,\nu) \in \mathbb{N} \times [0,1] - \{(0,0)\}$.
Then there exists a residual subset $\mathscr{\cR}$ of $C^{r,\nu}(M, G)$ such that for any $A\in\mathscr{G}$, the cocycle $(A,f)$ has at least one positive Lyapunov exponent
at $\mu$-a.e.\ point.
\end{corollary}

\section{Proofs}\label{s.proofs}

Here we review some intermediate results from \cite{Viana} in Subsections~\ref{ss:holonomies} and \ref{ss:disint},
then we recall some algebraic facts in Subsection~\ref{ss:groups},
and finally we prove Theorem~\ref{Viana1} 
in Subsection~\ref{ss.thmViana1}.

\subsection{Holonomies}\label{ss:holonomies}
In this and in the next subsection, we assume that $f$ is a $C^{1+\alpha}$-diffeomorphism of a Riemannian compact manifold $M$ preserving a non-atomic hyperbolic measure $\mu$ with local product structure, and that $A\in C^{r,\nu}(M, \mathrm{SL}(d, \mathbb{K}))$ for some $(r,\nu) \in \mathbb{N} \times [0,1] - \{(0,0)\}$ (and $\mathbb{K}=\mathbb{R}$ or $\mathbb{C}$).

A key insight from \cite{Viana} is that the vanishing of Lyapunov exponents of the cocycle $(A,f,\mu)$ implies the existence of a dynamical structure called \emph{stable and unstable holonomies}.

More concretely, let 
$f$ be a $C^{1+\alpha}$-diffeomorphism of a Riemannian compact manifold $M$ preserving a non-atomic hyperbolic measure with local product structure. 
Let $A\in C^0(M, \mathrm{SL}(d, \mathbb{K}))$ be a continuous linear cocycle. 

\begin{definition}\label{def:discrete.domination}
Given $N\ge 1$ and $\theta>0$,
let $ \mathscr{D}_A(N,\theta)$ denote the set of points $x\in M$ satisfying:
$$
\prod_{j=0}^{k-1} \big\|  A^{(N)}(f^{jN}(x))  \big\|  \; \big\|  A^{(N)}(f^{jN}(x))^{-1}  \big\| \le e^{k N \theta}
	 \quad\text{for all $k\in \mathbb N$.}
$$
We say that $\mathcal O$ is a \emph{holonomy block for $A$} if it is a compact subset of
$\mathcal H(K, \tau) \cap \mathscr{D}_A(N,\theta)$ for some constants $K$, $\tau$, $N$, $\theta$ satisfying $3\theta<\tau$.
\end{definition}

By~\cite[Corollary 2.4]{Viana}, if all Lyapunov exponents of $(A,f)$ vanish at $\mu$-almost every point then there exist holonomy blocks of measure arbitrarily close to $1$.

By~\cite[Proposition 2.5]{Viana}, the limits
$$
H^\mathrm{s}_{A, x, y} = H^\mathrm{s}_{x, y} \coloneqq \lim\limits_{n\to +\infty} A^{(n)}(y)^{-1} A^{(n)}(x) \quad \text{and} \quad H^\mathrm{u}_{A, x, z}= H^\mathrm{u}_{x, z}\coloneqq\lim\limits_{n\to -\infty} A^{(n)}(z)^{-1} A^{(n)}(x) \, ,
$$
called \emph{stable and unstable holonomies},
exist whenever $x$ belongs to a holonomy block $\mathcal{O}$, $y\in W^\mathrm{s}_{\mathrm{loc}}(x)$ and $z\in W^\mathrm{u}_{\mathrm{loc}}(x)$. 
These holonomy maps depend differentiably on the cocycle:

\begin{proposition}[{\cite[Lemma 2.9]{Viana}}]\label{p.holonomy-dependence-2} 
Given a periodic point $p$ in a holonomy block and points $y\in W^\mathrm{s}_{\mathrm{loc}}(p)$ and $z\in W^\mathrm{u}_{\mathrm{loc}}(p)$, the maps $B\mapsto H^\mathrm{s}_{B, p, y}$ and $B\mapsto H^\mathrm{u}_{B, z, p}$ from a small neighborhood $\mathcal{U}$ of $A$ to $\mathrm{SL}(d,\mathbb{C})$  are $C^1$, with derivatives:
\begin{eqnarray}\label{derivative of s-holo}
\partial_B H^\mathrm s_{B,p,y}: \dot{B}\mapsto &&\sum^\infty_{i=0}B^{(i)}(y)^{-1}[H^\mathrm s_{B, f^i(p), f^i(y)}B(f^i(p))^{-1}\cdot \dot{B}(f^i(p))\\&-&B(f^i(y))^{-1}\dot{B}(f^i(y))H^\mathrm s_{B,f^i(p),f^i(y)}]\cdot B^{(i)}(p)\nonumber\\\label{derivative of u-holo}
\partial_B H^\mathrm u_{B,z,p}:\dot{B}\mapsto &&\sum^{\infty}_{i=1}B^{(-i)}(p)^{-1}[H^\mathrm u_{B, f^{-i}(z),f^{-i}(p)}B(f^{-i}(z))\dot{B}(f^{-i}(z))\\&-&B(f^{-i}(p))\dot{B}(f^{-i}(p))H^\mathrm u_{B,f^{-i}(z),f^{-i}(p)}]\cdot B^{(-i)}(z)\nonumber
\end{eqnarray}
\end{proposition}

\begin{remark} In this statement, it is implicit the fact ensured by \cite[Corollary 2.11]{Viana} that the same holonomy block works for all $B\in\mathcal{U}$.
\end{remark}

Suppose $\mathcal{O}_i$, where $i=1,\dots,l$, are holonomy blocks of $(f,A)$ containing horseshoes $H_i$ associated to periodic points $p_i\in\mathcal{O}_i$ of minimal periods $\kappa_i$ and some homoclinic points $z_i\in\mathcal{O}_i$ of $p_i$, say $z_i\in W^\mathrm{u}_\mathrm{loc}(p_i)$ and $f^{q_i}(z_i)\in W^\mathrm{s}_\mathrm{loc}(p_i),\quad q_i>0$, such that $p_i$, $z_i\in \mathrm{supp}(\mu\mid \mathcal{O}_i\cap f^{-\kappa_i}(\mathcal{O}_i))$. By the remark above we know there is a neighborhood $\mathcal{U}$ of $A$ such that all the same holonomy blocks $\mathcal{O}_i$ (hence $p_i,z_i,q_i$) still work for any $B\in \mathcal{U}$. Then we have the following important lemma:
\begin{lemma}\label{gi submersion}
The map 
\begin{eqnarray*}
\Phi: \mathcal{U}&\to & G^{2l}\\
B&\mapsto & (g_{1,1}(B), \dots, g_{l,1}(B), g_{1,2}(B),\dots,  g_{l,2}(B))
\end{eqnarray*}
is a submersion at every $B\in\mathcal{U}$, where 
\begin{equation}
g_{i,1}(B)\coloneqq B^{(\kappa_i)}(p_i) \quad \mathrm{ and } \quad 
g_{i,2}(B)\coloneqq H^\mathrm{s}_{B,f^{q_i}(z_i),p_i} \circ B^{(q_i)}(z_i) \circ H^\mathrm{u}_{B,p_i,z_i}
\end{equation}
\end{lemma}
\begin{proof}
Fix $V_{z_i}$, $V_{p_i}$ some neighborhoods of $p_i$ and $z_i$. Without loss of generality we could assume $V_{z_i}, V_{p_i},i=1,\dots, l$ are small enough such that 
\begin{eqnarray}\nonumber
f^n(p_i)\cap V_{z_j}&=&\emptyset,\quad \forall i,j,n\\ \nonumber
f^n(p_i)\cap V_{p_j}&=&\emptyset\quad\text{ except if } i=j \text{ and } \kappa_i|n\\ \label{small perturbation region}
f^n(z_i)\cap V_{z_j}&=&\emptyset\quad\text{ except if } i=j \text{ and } n=0 
\end{eqnarray}

We claim that the derivative of map $\Phi$ is surjective at every point of $\mathcal{U}$, even when restricted to the
subspace of tangent vectors $\dot{B}$ supported on $\bigcup_i V_{p_i} \cup \bigcup_i V_{z_i}$. In fact for every $B\in \mathcal{U}$, for tangent vectors $\dot{B}$ supported on $\bigcup_i V_{p_i}\cup \bigcup_i V_{z_i}$ we will prove that the derivative of $\Phi$ has the following lower triangular form: 
\begin{equation}\label{derivative of Phi}
\partial_B\Phi^T (\dot{B})=(\partial_Bg_{1,1}(\dot{B}), \dots, \partial_Bg_{l,1}(\dot{B}), \quad\partial_Bg_{1,2}(\dot{B}),\dots,  \partial_Bg_{l,2}(\dot{B}))^T=\begin{pmatrix}
\partial\Phi_{1,1}&0\\
\ast&\partial \Phi_{2,2}
\end{pmatrix}\cdot \begin{pmatrix}
\dot{B}_p\\\dot{B}_z
\end{pmatrix}
\end{equation}
where $\dot{B}_p=(\dot{B}(p_1),\dots,\dot{B}(p_l))^T, \dot{B}_z=( \dot{B}(z_1),\dots,\dot{B}(z_l))^T$ and $\partial\Phi_{1,1}, \partial\Phi_{2,2}$ are two diagonal surjective linear maps.

By \eqref{small perturbation region}, we easily get 
\begin{eqnarray}\label{1st row of Phi'}
\partial_B(g_{i,1}(\dot{B}))=\begin{cases}
g_{i,1}(B)\cdot B(p_i)^{-1}\cdot \dot{B}(p_i),&\text{ if }\mathrm{supp}(\dot{B})\subset V_{p_i}, \\
0, &\text{ if } \mathrm{supp}(\dot{B})\subset V_{p_j},~j\neq i~\text{ or }V_{z_j}, 1\leq j\leq l.
\end{cases}
\end{eqnarray}
By $g_{i,2}$'s definition,
\begin{eqnarray}\label{derivative of g2}
\partial_B(g_{i,2}(\dot{B}))&=&\partial_BH^\mathrm{s}_{B,f^{q_i}(z_i),p_i}(\dot{B})\cdot B^{(q_i)}(z_i) \cdot H^\mathrm{u}_{B,p_i,z_i}\\\nonumber
&+& H^\mathrm{s}_{B,f^{q_i}(z_i),p_i}\cdot \partial_B B^{(q_i)}(z_i)(\dot{B})\cdot H^\mathrm{u}_{B,p_i,z_i}\\ \nonumber
&+& H^\mathrm{s}_{B,f^{q_i}(z_i),p_i}\cdot B^{(q_i)}(z_i)\cdot \partial_BH^\mathrm{u}_{B,p_i,z_i}(\dot{B})
\end{eqnarray}
By \eqref{derivative of s-holo}, \eqref{derivative of u-holo} and \eqref{small perturbation region}, for any $j$, 
\begin{equation}\label{vanishing derivative hs hu}
\partial_B H^\mathrm{s}_{B,f^{q_i}(z_i),p_i} (\dot{B})= \partial_BH^\mathrm{u}_{B,p_i,z_i}(\dot{B})=0 \text{ if }\mathrm{supp}(\dot{B})\subset V_{z_j}
\end{equation}
and 
\begin{equation}\label{Bq deriv}
\partial_B B^{(q_i)}(z_i)(\dot{B})=\begin{cases}B^{(q_i)}(z_i)\cdot B(z_i)^{-1}\cdot \dot{B}(z_i) &\text{ if }\mathrm{supp}(\dot{B})\subset V_{z_i}, \\
0, &\text{ if }\mathrm{supp}(\dot{B})\subset V_{z_j}, j\neq i.
\end{cases}
\end{equation}
Combining \eqref{derivative of g2}, \eqref{vanishing derivative hs hu} and \eqref{Bq deriv} we get 
\begin{equation}\label{short form deriv g2 }
\partial_B(g_{i,2}(\dot{B}))=\begin{cases}H^\mathrm{s}_{B,f^{q_i}(z_i),p_i}\cdot B^{(q_i)}(z_i)\cdot B(z_i)^{-1}\cdot \dot{B}(z_i)\cdot H^\mathrm{u}_{B,p_i,z_i} &\text{if }\mathrm{supp}(\dot{B})\subset V_{z_i}, \\
0, &\text{if }\mathrm{supp}(\dot{B})\subset V_{z_j},j\neq i.
\end{cases}
\end{equation}
Then by \eqref{short form deriv g2 }, \eqref{1st row of Phi'} and invertibility of $H^{\mathrm{s,u}}$ and $B$, we get \eqref{derivative of Phi}.
As explained before, Lemma \ref{gi submersion} follows.
\end{proof}

\color{black}
\subsection{Disintegrations}\label{ss:disint}

Let $f_A$ denote the induced \emph{projectivized cocycle}, that is, the skew-product map on $M\times \mathbb{P}^{d-1}(\mathbb{K})$ defined by $(x,[v]) \mapsto (f(x),[A(x)v])$.

By compactness of the projective space, the projectivized cocycle $f_A$ always has invariant probability measures $m$ on $M\times \mathbb{P}^{d-1}(\mathbb{K})$ projecting down to $\mu$ on $M$. Any such measure $m$ can be disintegrated (in an essentially unique way) into a family of measures $m_z$ on $\{z\}\times \mathbb{P}^{d-1}(\mathbb{K})$, $z\in M$, in the sense that $m(C)=\int m_z(C\cap (\{z\}\times \mathbb{P}^{d-1}(\mathbb{K}))) \; d\mu(z)$ for all measurable subsets $C\subset M\times \mathbb{P}^{d-1}(\mathbb{K})$: see \cite[Section 10.6]{Bogachev}.

As explained in Subsection~\ref{ss.hyperbolic-systems},  $(f,\mu)$ has hyperbolic blocks $\cH(K, \tau)$ of almost full $\mu$-measure.
Given a holonomy block $\mathcal{O}$ of positive $\mu$-measure inside a hyperbolic block $\mathcal{H}(K,\tau)$, $\delta>0$ sufficiently small (depending on $K$ and $\tau$) and a point $x\in\mathrm{supp}(\mu \mid \mathcal{O})$, we denote by $\cN_{x}(\cO,\delta)$,  $\cN^\mathrm{u}_{x}(\cO,\delta)$ and $\cN^\mathrm{s}_{x}(\cO,\delta)$ the 
subsets of $\cN_{x}(\delta)$,  $\cN^\mathrm{u}_{x}(\delta)$ and $\cN^\mathrm{s}_{x}(\delta)$
obtained by replacing $\mathcal{H}(K,\tau)$ by $\mathcal{O}$ in the definitions. 

The next result extracted from~\cite[Proposition 3.5]{Viana} says that the disintegration behaves in a rigid way when all Lyapunov exponents of the cocycle vanish. For simplicity, the action of a linear map $L$ on the projective space is also denoted by $L$.

\begin{proposition}\label{prop:disintegration} Suppose that all Lyapunov exponents of $(A,f)$ vanish at $\mu$-almost every point.

If $\cO$ is a holonomy block of positive $\mu$-measure, $\delta>0$ is sufficiently small and $x \in \supp(\mu\mid \cO)$, then every $f_A$-invariant probability
measure $m$ on $M\times \mathbb{P}^{d-1}(\mathbb{K})$ projecting down to $\mu$ on $M$ admits a disintegration $\{m_z: z\in M\}$ such that the function $\supp(\mu\mid \mathcal N_x(\cO,\delta)) \ni z\mapsto m_z$ is continuous in the weak$^*$ topology and, moreover,
$$
(H^\mathrm{s}_{y,z})_* m_{y} = m_z = (H^\mathrm{u}_{w,z})_* m_{w}
$$
for all $y,z,w \in \supp(\mu\mid \mathcal \mathcal N_x(\cO,\delta))$ with $y\in W^\mathrm{s}_{\mathrm{loc}}(z)$ and   $w\in W^\mathrm{u}_{\mathrm{loc}}(z)$.
\end{proposition}

We shall exploit the rigidity condition in the previous proposition through the following result extracted from \cite[Proposition 4.5]{Viana} ensuring the existence of holonomy blocks containing periodic points and some of its homoclinic points when all Lyapunov exponents vanish in a set of positive measure.

\begin{proposition}\label{prop:transversal-1} 
Suppose that all Lyapunov exponents of $(A,f)$ vanish at $\mu$-almost every point. Then for any $l>0$, there exists $l$ holonomy blocks $\mathcal{O}_i$, $1\leq i\leq l$, containing horseshoes $H_i$ associated to periodic points (with different orbits) $p_i\in\mathcal{O}_i$ of period $\pi(p_i)$ and some homoclinic points $z_i\in\mathcal{O}_i$ of $p_i$ such that $p_i, z_i\in \mathrm{supp}(\mu\mid \mathcal{O}_i\cap f^{-\pi(p_i)}(\mathcal{O}_i))$.
\end{proposition}

\subsection{Some facts about linear algebraic groups}\label{ss:groups} 

Recall that $G$ is an algebraic group of matrices satisfying the hypotheses listed at Subsection~\ref{ss:statement}. In this subsection we collect some algebraic facts that we will use.

The following property, shown by Breuillard \cite[Lemma 6.8]{Breuillard} (see also \cite[Lemma 7.7]{Aoun}) only needs the fact that $G$ is algebraic and semisimple (hypothesis \eqref{i.semisimple}):

\begin{proposition}\label{p.Breuillard}
There exists a proper algebraic subvariety $V\subset G\times G$ such that any pair of elements $(g_1,g_2) \in (G\times G)-V$ generates a Zariski dense subgroup of $G$.
\end{proposition}

The following is Proposition~3.2.15 in \cite{Zimmer}, and uses our hypotheses  \eqref{i.connected} and \eqref{i.irreducible}:

\begin{proposition}\label{p.Zimmer}
Let $m$ be a probability measure on the projective space $\mathbb{P}^{d-1}(\mathbb{K})$, and let $G_m$ be the set of elements of $G$ whose projective actions preserve $m$.
Then:
\begin{enumerate}[label=\rm(\roman*),ref=\roman*]
\item\label{i.compact} $G_m$ is compact, or
\item\label{i.proper} $G_m$ is contained in a proper algebraic subgroup of $G$.
\end{enumerate} 
\end{proposition}

\begin{remark}
Actually $G_m$ is an amenable subgroup, by Theorem 2.7 in \cite{Moore}, but we will not need this fact. 
\end{remark}

\begin{remark}\label{r.quadratic}
If $\mathbb{K}=\mathbb{R}$ then property \eqref{i.compact} in Proposition~\ref{p.Zimmer} actually implies (under our hypothesis \eqref{i.noncompact}) property \eqref{i.proper}. Indeed, by a well-known fact \cite[Proposition~4.6]{Knapp}, every compact subgroup of $\mathrm{SL}(d,\mathbb{R})$ preserves a positive definite quadratic form, and in particular is $\mathbb{R}$-algebraic. \footnote{These implications fail in the complex case; for example the compact group $\mathrm{SU}(d)$ is Zariski-dense in $\mathrm{SL}(d,\mathbb{C})$.}
\end{remark}

Recall that a subset of $\mathbb{R}^n$ is called semi-algebraic if it is defined by finitely many polynomial inequalities\footnote{I.e., a semi-algebraic set is an element of the smallest Boolean ring of subsets of $\mathbb{R}^n$ containing all subsets of the form $\{(x_1,\dots,x_n)\in\mathbb{R}^n: P(x_1,\dots, x_n)>0\}$ with $P\in\mathbb{R}[X_1,\dots,X_n]$.}, and the dimension of a semi-algebraic set is the maximal local dimension near regular points (see, e.g., \cite{strat-2}). 

\begin{lemma}\label{l.Matheus}
Suppose $\mathbb{K}=\mathbb{C}$. 
Then there is a semi-algebraic set $W \subset G \times G$ of positive codimension such that 
for any pair of elements $(g_1,g_2) \in (G\times G)-W$, the group they generate is not contained in any compact subgroup of $G$.
\end{lemma}

\begin{proof}
Let $K$ be a maximal compact subgroup of $G$. We think of $G$ as a complexification of $K$: in particular, the Lie algebra of $G$ is the tensor product over $\mathbb{R}$ of $\mathbb{C}$ and the Lie algebra of $K$, and, \emph{a fortiori}, $\textrm{dim}_{\mathbb{R}}(G) = 2\cdot\textrm{dim}_{\mathbb{R}}(K)$. 

Consider a maximal Abelian subgroup $A$ of $G$ and the corresponding decomposition $G=KAK$ coming from the diffeomorphism $K\times\textrm{exp}(\mathfrak{p})\to G$ where $\mathfrak{p}=\bigcup\limits_{k\in K} \textrm{Ad}(k)\cdot\mathfrak{a}$, $\textrm{Ad}(.)$ denotes the adjoint action, and $\mathfrak{a}$ is the Lie algebra of $A$. Note that $\textrm{dim}_{\mathbb{R}}(G) > \textrm{dim}_{\mathbb{R}}(K) + \textrm{dim}_{\mathbb{R}}(A)$ (as one can infer, for instance, from the Killing-Cartan classification of simple complex Lie groups via Dynkin diagrams: see, e.g., \cite{Knapp} for more details). 

Define $\Phi:K\times A\times K\times K\to G\times G$ by $\Phi(k,a,u,v) = (kaua^{-1}k^{-1}, kava^{-1}k^{-1})$. Note that $\Phi$ is a polynomial map between semi-algebraic sets. Hence, its image $W:=\Phi(K\times A\times K\times K)$ is a semi-algebraic set (by Tarski-Seidenberg theorem) of dimension  $\textrm{dim}_{\mathbb{R}}(W)\leq \textrm{dim}_{\mathbb{R}}(K\times A\times K\times K) = (\textrm{dim}_{\mathbb{R}}(K)+\textrm{dim}_{\mathbb{R}}(A))+2\cdot\textrm{dim}_{\mathbb{R}}(K)$.

Since $\textrm{dim}_{\mathbb{R}}(G) > \textrm{dim}_{\mathbb{R}}(K)+\textrm{dim}_{\mathbb{R}}(A)$ and $\textrm{dim}_{\mathbb{R}}(G\times G) = 2\cdot \textrm{dim}_{\mathbb{R}}(G) = \textrm{dim}_{\mathbb{R}}(G)+2\cdot \textrm{dim}_{\mathbb{R}}(K)$, it follows that $\textrm{dim}_{\mathbb{R}}(W)\leq (\textrm{dim}_{\mathbb{R}}(K)+\textrm{dim}_{\mathbb{R}}(A))+2\cdot\textrm{dim}_{\mathbb{R}}(K) < \textrm{dim}_{\mathbb{R}}(G\times G)$.

In summary, $W$ is a semi-algebraic subset of $G\times G$ of positive codimension. Therefore, the proof of the lemma will be complete once we show that if $(g_1,g_2)\in G\times G$ generates a group contained in a compact subgroup of $G$, then $(g_1, g_2)\in W$. In this direction, we observe that $K$ is a maximal compact subgroup of $G$, so that if the closure of the subgroup generated by $g_1$ and $g_2$ is compact, then there exists $g\in G$ such that $g_1, g_2\in gKg^{-1}$, say $g_1=gxg^{-1}$ and $g_2=gyg^{-1}$ with $x, y\in K$. On the other hand, the decomposition $G=KAK$ allows us to write $g=kak'$ for some $k, k'\in K$ and $a\in A$. It follows that 
$$(g_1, g_2) = (gxg^{-1}, gyg^{-1}) = (ka(k'xk'^{-1})a^{-1}k^{-1}, ka(k'yk'^{-1})a^{-1}k^{-1}) = \Phi(k,a,u,v)$$
with $k\in K$, $a\in A$, $u=k'xk'^{-1}\in K$ and $v=k'yk'^{-1}\in K$, i.e., $(g_1,g_2)\in W$. This completes the proof. 
\end{proof}

By combining the previous results, we deduce the following:

\begin{corollary}\label{corol}
There is a semi-algebraic set $Z \subset G \times G$ of positive codimension such that 
no pair of elements $(g_1, g_2)\in (G\times G)-Z$ admits a common invariant measure on projective space $\mathbb{P}^{d-1}(\mathbb{K})$.
\end{corollary}

\begin{proof}
If $\mathbb{K}=\mathbb{R}$ then we let $Z = V$ be the proper algebraic subvariety of $G\times G$ described in Proposition~\ref{p.Breuillard} above.
Otherwise, if $\mathbb{K}=\mathbb{C}$ then we let $Z = V \cup W$ where $W$ is given by Lemma~\ref{l.Matheus}.

Now consider a pair of elements $(g_1, g_2)\in G\times G$ that admit a common invariant measure $m$ on $\mathbb{P}^{d-1}(\mathbb{K})$. If $\mathbb{K}=\mathbb{C}$ then $(g_1, g_2)$ belongs to either $W$ or $V$, according to which property \eqref{i.compact} or \eqref{i.proper} holds in Proposition~\ref{p.Zimmer}.
If $\mathbb{K}=\mathbb{R}$ then by Remark~\ref{r.quadratic} we know that property \eqref{i.proper} holds, so $(g_1, g_2) \in V$.
\end{proof}


\subsection{Proof of Theorem~\ref{Viana1}}\label{ss.thmViana1}

Let $G$ be an algebraic group of matrices satisfying the hypotheses listed at Subsection~\ref{ss:statement}.
Let $f$ be a $C^{1+\alpha}$-diffeomorphism of a compact manifold $M$.
Let $\mu$ be a $f$-invariant ergodic hyperbolic non-atomic probability measure with local product structure.

Let $A\in C^{r,\nu}(M,G)$ be a cocycle whose Lyapunov exponents vanish at $\mu$-almost every point. To prove Theorem~\ref{Viana1}, we only need to prove that for any $l>0$ there exists a neighborhood $\mathcal{U}\subset C^{r,\nu}(M,G)$ of $A$ such that the cocycles in $\mathcal{U}$ with vanishing Lyapunov exponents are contained in a Whitney stratified set with codimension $\geq l$. 

By Proposition~\ref{prop:transversal-1}, we can find $l$ holonomy blocks $\mathcal{O}_i$ of positive $\mu$-measure containing horseshoes $H_i$ associated to distinct periodic points $p_i\in\mathcal{O}_i,1\leq i\leq l$ of minimal periods $\kappa_i$, and some homoclinic points $z_i\in\mathcal{O}_i$ of $p_i$, $z_i\in W^\mathrm{u}_\mathrm{loc}(p_i), f^{q_i}(z_i)\in W^\mathrm{s}_\mathrm{loc}(p_i)$ such that $p_i$, $z_i\in \mathrm{supp}(\mu\mid \mathcal{O}\cap f^{-\kappa_i}(\mathcal{O}_i))$ and $q_i>0$. Moreover the same $\mathcal{O}_i, p_i,z_i,q_i$ work for any $B$ in a small neighborhood $\mathcal{U}$ of $A$.

Then by Proposition~\ref{prop:disintegration}, for any $A'\in \mathcal{U}$ with vanishing Lyapunov exponents, for any $1\leq i\leq l$ the projective actions of the matrices 
$$
g_{i,1}(A')\coloneqq A'^{(\kappa_i)}(p_i) \quad \mathrm{ and } \quad 
g_{i,2}(A')\coloneqq H^\mathrm{s}_{A',f^{q_i}(z_i),p_i} \circ A'^{(q_i)}(z_i) \circ H^\mathrm{u}_{A',p_i,z_i}
$$ 
preserve a common probability measure $m_{p_i}(A')$ on $\mathbb{P}^{d-1}(\mathbb{K})$.
\color{black}
Thus, for any $i$, the pair $(g_{i,1}(A'),g_{i,2}(A'))$ belongs to the semi-algebraic set $Z$ of positive codimension in $G \times G$ given by Corollary~\ref{corol}.
Recall (see \cite{strat-2}) that:
\begin{itemize}
\item semi-algebraic sets are Whitney stratified; 
\item products of Whitney stratified sets is Whitney stratified and codimensions add;
\item pre-images of Whitney stratified sets under submersions are Whitney stratified, and codimension is preserved.
\end{itemize} 
Therefore by Lemma \ref{gi submersion} we conclude that all such $A'$ lie in a Whitney stratified subset of codimension $\geq l$. This completes the proof. \qed

\end{document}